\documentclass[12pt]{article}
\usepackage[T1]{fontenc}
\usepackage[a4paper]{geometry}
\usepackage{verbatim}
\usepackage{mathtools}
\usepackage{dsfont}
\usepackage{amsmath, amssymb, amsfonts, amsthm}
\usepackage[numbers]{natbib}
\usepackage{hyperref}
\usepackage{authblk}

\usepackage{url}
\usepackage{nicefrac}
\usepackage{graphicx}
\usepackage{doi}
\usepackage{xcolor}
\usepackage{bbm}
\usepackage{comment}
\usepackage[shortlabels]{enumitem}
\usepackage[leqno,fleqn,intlimits]{empheq}
\usepackage{mathrsfs}

\usepackage{soul}

\newcommand{\N}{\mathbb{N}} 
\newcommand{\Z}{\mathbb{Z}}

\newcommand{\restr}[2]{\left.\kern-\nulldelimiterspace #1 \vphantom{|} \right|_{#2}}
\newcommand{\E}{{\mathbb E}}

\newcommand*{\diff}{\mathop{}\!\mathrm{d}}
\newcommand{\SMALL}{\textstyle}
\newcommand{\wt}{\widetilde}
\newcommand{\1}{\mathbbm 1}

\renewcommand{\P }{\mathbb P}
\renewcommand{\epsilon}{\varepsilon}
\renewcommand{\phi}{\varphi}

\DeclareMathOperator*{\Po}{Po}

\numberwithin{equation}{section}

\theoremstyle{definition}

\newtheorem{thm}{Theorem}[section]
\newtheorem{lem}[thm]{Lemma}

\newtheorem{prop}[thm]{Proposition}
\newtheorem{remark}[thm]{Remark}

\newtheorem*{defin*}{Definition}
\newtheorem*{thm*}{Theorem}
\newtheorem*{cor*}{Corollary}
\newtheorem*{prop*}{Proposition}
\newtheorem*{lem*}{Lemma}
\newtheorem*{remark*}{Remark}
\newtheorem*{eg*}{Example}

\begin{document}
\title{A threshold for Poisson behavior of non-stationary product measures}
\date{}
\author{Michael Hochman\thanks{michael.hochman@mail.huji.ac.il}\enspace and Nicolò Paviato\thanks{nicolo.paviato@mail.huji.ac.il}}
\maketitle
\begin{center}
    \vspace{-4em} 
    \large 
    \textit{Einstein Institute of Mathematics, The Hebrew University of Jerusalem, Jerusalem, Israel.}
    
    \medskip 
    
    December 4, 2025

    \medskip
\end{center}

\begin{abstract}
Let $\gamma_{n}= O (\log^{-c}n)$ and let $\nu$ be the infinite product measure
whose $n$th marginal is Bernoulli$(1/2+\gamma_{n})$.
We show that $c=1/2$ is the threshold, above which $\nu$-almost every point is simply Poisson generic in the sense of Peres-Weiss, and below which
this can fail. This provides a range in which $\nu$ is singular with
respect to the uniform product measure, but $\nu$-almost every point is simply Poisson generic.
\end{abstract}

\section{\label{sec:Introduction}Introduction}

Many notions of `randomness' have been proposed for individual
infinite sequences $x\in\{-1,1\}^{\mathbb{N}}$. The simplest one is normality, introduced by Borel~\cite{Bor09} more than a hundred years ago,
which in this context means that every finite pattern $\omega\in\{-1,1\}^{k}$
appears in $x$ with asymptotic frequency $2^{-k}$, as would occur
if $x$ were a typical point for the ``uniform'' product measure
$\mu^{\N}=\prod_{n=1}^{\infty}\bigl[(1/2)\delta_{1}+(1/2)\delta_{-1}\bigr]$.

Here, we shall be concerned with the notion of simply Poisson genericity,
which was introduced by Z. Rudnik and is defined as follows. Given
$x\in\{-1,1\}^{\mathbb{N}}$, let $W_{k}$ be a uniformly sampled
random word in $\{-1,1\}^{k}$ and let $M_{k}^{x}$ denote the
(random) number of appearances of $W_{k}$ in $x$ up to time $2^{k}$:
\[
    M_k^x=\#\{1\leq j\leq2^{k}\mid x_{j}\ldots x_{j+k-1}=W_{k}\}.
\]
Then $x$ is \textit{simply Poisson generic} if $M_k^x$ converges in distribution
to a Poisson random variable with mean one (briefly, $M_k^x\xrightarrow{d}\Po(1)$), that is 
\[
    \lim_{k\rightarrow\infty}\P (M_k^x=n)=\frac{1}{e}\cdot\frac{1}{n!},
\]
for all $n\in\Z_{\ge0}$.
Throughout this paper, we sometimes omit the term ``simply'' and call this  property \textit{Poisson normality} for short. Note that the unqualified term Poisson generic has a stronger meaning in~\cite{Wei20}.

In unpublished work (see~\cite{Wei20}), Yuval Peres and Benjamin
Weiss proved the following.
\begin{itemize}
    \item If $x$ is Poisson generic, then it is normal.
    \item Almost every $x$ for the uniform product measure on $\{-1,1\}^{\mathbb{N}}$
    is Poisson normal.
    \item Normality does not imply Poisson normality.
\end{itemize}
For a long time it was an open problem to exhibit explicit examples
of simply Poisson generic sequences, but recently an example over larger alphabets was given
by~\cite{BecSac23}. We also mention~\cite{AlvBecCesMerPerWei24} which extends almost sure Poisson genericity to settings with infinite alphabets and exponentially mixing probability measures.

Since simply Poisson generic points are normal, the ergodic theorem tells
us that $\mu^{\N}$ is the only ergodic shift-invariant measure on $\{-1,1\}^{\mathbb{N}}$
that can be supported, or even give positive mass, to simply Poisson generic
points. However, one may ask about non-shift-invariant measures. The most
natural class to consider is that of product measures, 
\[
    \nu=\prod_{n=1}^{\infty}\nu_{n},
\]
where $\nu_{n}$ are non-trivial measure on $\{-1,1\}$. We parametrize
the $\nu_{n}$ using the sequence
\[\SMALL
    \gamma_{n}=\frac{1}{2}-\nu_{n}(\{-1\}),
\]
so $\nu_{n}=((1/2)-\gamma_{n})\delta_{-1}+((1/2)+\gamma_{n})\delta_{1}$.
Observe that
\begin{enumerate}[(i)]
\item\label{item:cesaro_convergence} If $\nu_{n}\rightarrow \textrm{uniform measure on $\{-1,+1\}$}$ (equivalently, $\gamma_{n}\rightarrow0$),
then $\nu$-a.e.\@~point is normal. In fact, $\nu$-almost-sure normality is characterized
by Cesaro convergence, $N^{-1}\sum_{n=1}^{N}\gamma_{n}\rightarrow0$ as $N\to\infty$. 
Since Poisson normality implies normality, the latter is a necessary condition
for $\nu$ to be supported on simply Poisson generic points.
\item By a theorem of Kakutani~\cite{Kak48}, $\nu$ and $\mu^{\N}$ are equivalent
as measures if and only if $\sum_{n=0}^{\infty}\gamma_{n}^{2}<\infty$.
In this case, $\nu$-a.e.~\@$x\in\Omega^{\N}$
is simply Poisson generic, because this is true for $\mu^{\N}$.
\end{enumerate}
Our main result is to identify a threshold, stated in terms of the
decay of $(\gamma_{n})$, which separates product measures that are
supported on simply Poisson generic points, from those that are not. It turns
out that this decay rate is far slower than the rate in Kakutani's
theorem, so we obtain product measures $\nu$ that are singular with
respect to $\mu^{\N}$, but are nonetheless supported on simply Poisson generic points. This threshold is tight.
\begin{thm}\label{thm:main_thm} 
    Suppose that 
    $\gamma_n\in(-1/2,1/2)$ and $\nu$ is the corresponding product measure. If $\gamma_n=  O (\log^{-(1/2+\delta)}n)$ for some $\delta>0$, then $\nu$-almost every~$x\in\Omega^{\N}$ is simply Poisson generic. 
    On the other hand, if $\gamma_n=\log^{-(1/2-\delta)}n$ for all large~$n$, then 
    $\nu$-almost every~$x\in\Omega^{\N}$ is \textbf{not} simply Poisson generic.
\end{thm}

\begin{remark}
    We have stated the theorem for Poisson normality for simplicity, but it holds also for the stronger notion of Poisson genericity found in~\cite{Wei20}. Furthermore, the convergence result in Theorem~\ref{thm:main_thm} remains valid for sequences over finite alphabets $\{0,1,\dots,b-1\}$. In this broader context, the definition of Poisson normality counts the occurrences of a uniformly sampled word $W_k\in\{0,1,\dots,b-1\}^k$ within the first $b^k$ digits of a sequence $x\in\{0,1,\dots,b-1\}^\N$. For $\ell=0,\dots,b-1$, the associated measures are defined as $\nu_n(\{\ell\})=1/b+\gamma_n^{(\ell)}$, where$\{\gamma_{n}^{(\ell)}\}_{n\ge1}$ satisfies $\sum_{\ell=0}^{b-1}\gamma_n^{(\ell)}=0$ and $\gamma_n^{(\ell)}\in\bigl(-(b-1)/b,(b-1)/b\bigr)$ . The following proofs can be adapted to this setup to show that, assuming that $\max_{0\le \ell\le b-1}\gamma_n^{(\ell)}= O (\log^{-(1/2+\delta)}n)$, then $\nu$-a.e.\@~$x$ is simply Poisson generic.
\end{remark}

The remainder of the paper is organized as follows: in the next section we summarize our notation, in Section~\ref{sec:convergence} we prove the convergence result in Theorem~\ref{thm:main_thm}, while in Section~\ref{sec:non_conv} we establish tightness.

\section{\label{sec:Setup}Setup and notation}\label{sec:notation}

We let $\mathbb{N}=\{1,2,3,\ldots\}$ and for $n\in\mathbb{N}$ set $[n]=\{1,\ldots,n\}$.
Given a sequence $(\gamma_{n})_{n\in\N}$ taking values in $(-1/2,1/2)$ and $\Omega=\{-1,1\}$, we define the product measure $\nu$ on $\Omega^{\N}$
by 
\[
    \nu=\prod_{n=1}^{\infty}\nu_{n},\qquad\text{where }\nu_{n}(\{1\})=\frac{1}{2}+\gamma_{n}\qquad\text{and}\qquad\nu(\{-1\})=\frac{1}{2}-\gamma_{n}.
\]
Let $\mu^k$ denote the uniform product measure on $\Omega^k$, 
and consider $\P_k =\nu\times \mu^k$ defined on $\Omega^{\N}\times\Omega^k$. Denote by $\E_k $ the corresponding expectation.

For $1\leq j\leq2^{k}$, define the indicator random variables $I_{j}:\Omega^{\N}\times\Omega^k\rightarrow\{0,1\}$
by
\begin{equation}\label{eq:Indicator}
    I_{j}(x,\omega)=\left\{ \begin{array}{cc}
    1 & x_{j}\ldots x_{j+k-1}=\omega,\\
    0 & \text{otherwise},
    \end{array}\right.
\end{equation}
and $M_{k}:\Omega^{\N}\times\Omega^k\rightarrow\Z_{\ge0}$
by 
\begin{equation}\label{eq:Mk}\SMALL
    M_{k}(x,\omega)=\#\{1\leq i\leq2^{k}\mid x_{i}\ldots x_{i+k-1}=\omega\}=\sum_{j\in[2^k]}I_{j}(x,\omega).
\end{equation}

For $\omega\in\Omega^k$ and $j,k\ge1$, we introduce the quantity 
\begin{equation}\label{eq:Pj}
    P_{j,k}(\omega)=\prod_{i=1}^{k}(1+2\omega_{i}\gamma_{i+j-1}).
\end{equation}
Sometimes, we think of $P_{j,k}$ as a random variable on $\Omega^{\N}\times\Omega^k$.
By the independence of the random variables $\{\omega_i\gamma_{i+j-1}\}_{1\le i\le k}$, we point out that
\begin{equation}\label{eq:expectation_Pj}\SMALL
    \E_k [P_{j,k}]=\prod_{i=1}^k\E_k [1+2\omega_i\gamma_{i+j-1}]=1.
\end{equation}
We also note that for any fixed $\omega\in\Omega^k$,
\begin{equation}\label{eq:P(Ij=1)}\SMALL
    \P_k \bigl(I_j=1|\{\omega\}\bigr)=
    \prod_{i=1}^{k}\bigl(\frac{1}{2}+\omega_{i}\gamma_{i+j-1}\bigr)
    =2^{-k}P_{j,k}(\omega).
\end{equation}

Observe that, for any fixed $x\in\Omega^{\N}$, there is a unique $\omega\in\Omega^k$ such that
$I_j(x,\omega)=1$, and the probability of this $\omega$, like all
others, is $2^{-k}$; thus,
$\E_k [I_{j}]=2^{-k}$.
When $|i-j|\ge k$, the variables $I_j$ and $I_i$ are independent conditionally to $\omega\in\Omega^k$. However, the independence fails if we do not condition on $\omega$, since 
\[\SMALL
\E_k [I_iI_j]=2^{-k}\sum_{\omega\in\Omega^k}\nu\bigl(x:I_j(x,\omega)I_i(x,\omega)=1\bigr)=
2^{-3k}\sum_{\omega\in\Omega^k}P_{j,k}(\omega)P_{i,k}(\omega)
\]
is different from $\E_k [I_j]\E_k [I_i]=2^{-2k}$.

\section{\label{sec:convergence}Convergence to Poisson}

Let $\gamma_n\in(-1/2,1/2)$ be such that $\gamma_n= O (\log^{-(1/2+\delta)}n)$ for some $\delta>0$. Without loss of generality (decreasing $\delta$ if necessary), we assume that there is $n_0\ge1$ such that  
\begin{equation*}\label{eq:decay_rate}
|\gamma_n|\le\log^{-(1/2+\delta)}n,\qquad \text{for all } n\ge n_0.
\end{equation*}

We consider $M_k$ defined in~\eqref{eq:Mk} on the probability space  $(\Omega^\N\times\Omega^k,\P_k )$; the main result of this section is that~$M_{k}$  converges in distribution to a Poisson random variable with mean one.

\begin{prop}\label{prop:annealed}
\label{prop:convergence}
We have that $M_k\xrightarrow{d}\Po(1)$ as $k\to\infty$.
\end{prop}

This is commonly referred to as the \textit{annealed} case, because it involves a coupled probability space. By contrast, the \textit{quenched} scenario refers to an almost sure result on the probability space~$\Omega^\N$, corresponding precisely to the convergence statement of Theorem~\ref{thm:main_thm}.
The following proposition establishes the connection between annealed and quenched results.

\begin{prop}\label{prop:quenched}
\label{prop:quenched-implies-anealed}If $M_k\xrightarrow{d}\Po(1)$, then $\nu$-a.e.\@~$x\in\Omega^{\N}$ is
simply Poisson generic.
\begin{proof}

Using that $\nu$ is a product measure, this proof follows the same argument of Peres and Weiss, found in~\cite[Proof of Theorem 1]{AlvBecMer23}. The main tools are McDiarmid's inequality~\cite{McD89} and the Borel-Cantelli lemma.
\end{proof}
\end{prop}

By Proposition~\ref{prop:quenched} the convergence result in Theorem~\ref{thm:main_thm} follows from Proposition \ref{prop:annealed}; hence, the remainder of this section is dedicated to proving Proposition~\ref{prop:annealed}.

\subsection{A general convergence theorem applied to our setting}

To prove Proposition~\ref{prop:convergence}, we rely on a general
result on Poisson approximation,~\cite[Theorem 1.A]{BarHolJan92}  derived by the Chen-Stein method, which provides a bound on the total variation distance $d_{\text{TV}}$ (see the reference above for a definition).
We note that convergence
in total variation implies convergence in distribution. 
Given a family of random variables $\{X_j\}_{j\in J}$, we denote by $\sigma(X_j:j\in J)$ for the $\sigma$-algebra generated by such a family, that is the smallest $\sigma$-algebra for which each of the $X_j$ is measurable.

\begin{thm}\label{thm:Poi_approx}
Let $I_{1},\ldots,I_{n}$ be indicator random
variables and $S=\sum_{j\in[n]}I_{j}$. For every $j\in[n]$, let
there be given a partition $\Gamma_{j}^{s},\Gamma_{j}^{w}\subseteq[n]$
of $[n]\setminus\{j\}$, let 
\[\SMALL
    \lambda=\sum_{j\in[n]}\E [I_{j}],
\]
and let
\begin{equation*}\label{eq:etaj}
    \eta_{j}=\E \bigl|\E [I_{j}|\sigma(I_{i}:i\in\Gamma_{j}^{w})]-\E [I_{j}]\bigr|.
\end{equation*}
Then,
\[\begin{split}
    d_{TV}(S,\Po(\lambda))\le & \min\{1,\lambda^{-1}\}\biggl(\sum_{j\in[n]}\bigl(\E [I_{j}]^{2}+\sum_{i\in\Gamma_{j}^{s}}(\E [I_{j}]\E [I_{i}]+\E [I_{j}I_{i}])\biggr)\\
    & +\min\{1,\lambda^{-1/2}\}\sum_{j\in[n]}\eta_{j}.
\end{split}\]
\end{thm}

The sets $\Gamma_j^s,\Gamma_j^w$ partition the variables into those that are  strongly and weakly correlated with $I_j$, respectively. This is the meaning of the superscripts: "s" for strong and "w" for weak.

For a fixed $k$, we apply this with $n=2^{k}$, the indicators $I_{1},\ldots,I_{2^{k}}$ from~\eqref{eq:Indicator},
and $S=M_{k}=\sum_{j\in[2^k]}I_{j}$. Recall from Section~\ref{sec:Setup}
 that $\E_k [I_{j}]=2^{-k}$ and so
$\lambda=\sum_{j\in[2^{k}]}\E_k [I_{j}]=1.$
For $j\in[2^{k}]$ we let 
\begin{equation}\label{eq:strong_weak_dependence}
    \Gamma_{j}^{s}=\{n\in[2^{k}]\setminus\{j\}:|n-j|<k\}
    \qquad\text{and}\qquad
    \Gamma_{j}^{w}=[2^{k}]\setminus(\Gamma_{j}^{s}\cup\{j\}).
\end{equation}
Theorem~\ref{thm:Poi_approx} yields that 
\[
    d_{TV}(M_{k},\Po(1))\le\underset{A_{k}}{\underbrace{2^{-2k}\sum_{j\in[2^{k}]}(1+|\Gamma_{j}^{s}|)}}+\underset{B_{k}}{\underbrace{\sum_{j\in[2^{k}]}\sum_{i\in\Gamma_{j}^{s}}\E_k [I_{j}I_{i}]}}+\underset{C_{k}}{\underbrace{\sum_{j\in[2^{k}]}\eta_{j}}}
\]
In order to conclude that $M_{k}\xrightarrow{d}\Po(1)$, we
will show that each of the positive terms $A_{k},B_{k},C_{k}$ tend to zero as $k\to\infty$. 

\subsection[]{$A_{k}\rightarrow0$}

This is simple: by $|\Gamma_{j}^{s}|\le 2k$, we have $A_{k}\leq2^{-k}(1+2k)\to0$.

\subsection[]{$B_{k}\rightarrow0$}
\begin{lem}
\label{lem:expectation} There exists $j_{0}\in\N$
such that $\E_k [I_{i}I_{j}]<2^{-3k/2}$ for all $j_{0}\leq i<j\leq2^{k}$
satisfying $0<j-i< k$.
\end{lem}

\begin{proof}
Since $(\gamma_n)$ is a null sequence, we let $j_{0}$ be such that $1+2\gamma_{n}<2^{1/4}$ for all $n\ge j_0$, and let $i$, $j$ be as in the statement. Arguing as in the proof of~\citep[Lemma 1]{AlvBecMer23},
let
\[
    \Omega^k_{i,j}=\{\omega\in\Omega^k\mid(\omega_{1},\ldots,\omega_{k-(j-i)})=(\omega_{j-i},\ldots,\omega_{k})\},
\]
and note that a word $\omega\in\Omega^k$ can satisfy $I_{i}(x,\omega)I_{j}(x,\omega)=1$
for some $x\in\Omega^{\N}$, only if $\omega\in\Omega^k_{i,j}$.
The elements of $\Omega^k_{i,j}$ are in bijection with their prefix
of length $j-i$, so $\mu^k(\Omega^k_{i,j})=2^{-k+(j-i)}$.

For a fixed $\omega\in\Omega^k_{i,j}$, we define $\wt{\omega}\in\Omega^{k+(j-i)}$
as the juxtaposition of two copies of $\omega$, namely
    $\wt{\omega}_h=\omega_h$ if $h\in\{1,\dots,k\}$, and $\wt{\omega}_h=\omega_{h-(j-i)}$ if $h\in\{k+1,\dots,k+(j-i)\}$.
By~$i\ge j_0$,
\begin{equation*}\begin{split}\SMALL
    \nu(x:I_{i}(x,\omega)I_{j}(x,\omega)=1)&\SMALL=
    \prod_{h=i}^{j+k-1}\bigl(1/2+\wt{\omega}_{h-i+1}\gamma_{h}\bigr)\\
    &\SMALL\le    2^{-(k+j-i)}\prod_{h=i}^{j+k-1}\bigl(1+2\gamma_{h}\bigr)\\
    &\SMALL\le 2^{-(k+j-i)}2^{1/4(k+j-i)}.
\end{split}\end{equation*}
Using that $k+j-i<2k$, it follows that 
\[
    \nu(x:I_{i}(x,\omega)I_{j}(x,\omega)=1)\le 2^{-(k/2+j-i)}.
\]
So, 
\begin{align*}
    \E_k [I_{i} I_{j}] & \SMALL=\int_{\Omega^k_{i,j}}\nu(x:I_{i}(x,\omega)I_{j}(x,\omega)=1)\diff\mu^k(\omega)\\
     & \leq\mu^k(\Omega^k_{i,j})2^{-(k/2+j-i)}=2^{-3k/2},
\end{align*}
completing the proof.
\end{proof}
To conclude the proof that $B_{k}\rightarrow0$, we use that $\E_k [I_{j}]=2^{-k}$
to get
\[
    \E_k [I_{i}I_{j}]\leq\E_k [I_{j}]=2^{-k}.
\]
Therefore, with $j_{0}$ as in Lemma~\ref{lem:expectation}, 
\begin{align*}
    B_{k} & =\sum_{j\in[2^{k}]}\sum_{i\in\Gamma_{j}^{s}}\E_k [I_{i}I_{j}]\\
    & \leq\sum_{j=1}^{j_{0}-1}k\E_k [I_{i}I_{j}]+\sum_{j=j_{0}}^{2^{k}-1}k2^{-3k/2}\\
    & \leq j_{0}k2^{-k}+k2^{-k/2},
\end{align*}
and $B_{k}\rightarrow0$ follows.

\begin{remark}
    The arguments used so far do not rely on the specific rate at which~$(\gamma_n)$ decays to zero. This property becomes crucial in the next subsection.
\end{remark}

\subsection[]{$C_{k}\to0$} 
Let $P_{j,k}$ be as in~\eqref{eq:Pj}.
The main step to prove $C_k\to0$ is the following.

\begin{prop}\label{prop:Pj_uniform}
	Let $\epsilon>0$. Then $\E_k\bigl[|P_{j,k}-1|\bigr]\to 0$ uniformly in $j\ge 2^{\epsilon k}$ as $k\to\infty$.

\begin{proof}
As $k\to\infty$, the decay rate for $(\gamma_n)$ yields uniformly in $j\ge 2^{\epsilon k}$ that
	\[
	0\le \gamma_{j}^2\le\log^{-(1+2\delta)}j
	\le \log^{-(1+2\delta)}(2^{k\epsilon})
	= O (k^{-(1+2\delta)}).
	\]
So,
\begin{equation}\label{eq:decay}\SMALL
    \sum_{i=1}^k\gamma_{i+j-1}^2= O \bigl(k^{-2\delta}\bigr).
\end{equation} 
By a first order expansion around $x=0$ of $f(x)=\log(1+x)$ and~\eqref{eq:decay}, for $j>2^{\varepsilon k}$ we have
\begin{equation}\label{eq:P_{j,k}}
\begin{split}
    P_{j,k}(\omega)\SMALL=
    \exp\Bigl\{\sum_{i=1}^k\log\bigl(1+2\omega_i\gamma_{i+j-1}\bigr)\Bigr\}
    =\exp\Bigl\{2\sum_{i=1}^k\omega_i\gamma_{i+j-1}
    + O (k^{-2\delta})\Bigr\}.
\end{split}
\end{equation}

For a fixed $\theta\in(0,1/2)$, we define 
\[\SMALL
A_{k,j}^\theta=\bigl\{\omega\in \Omega^k:
\bigl|\sum_{i=1}^k\omega_i\gamma_{i+j-1}\bigr|\le
\bigl(\sum_{i=1}^k\gamma_{i+j-1}^2\bigr)^{1/2-\theta}
\bigr\}.
\]
By~\eqref{eq:decay}, we have uniformly in
$\omega \in A_{k,j}^\theta$ and $j\ge 2^{\epsilon k}$ that
$\sum_{i=1}^k\omega_i\gamma_{i+j-1}=O(k^{-\delta(1-2\theta)})$. So, identity~\eqref{eq:P_{j,k}} yields that 
$P_{j,k}(\omega)=1+o(1)$ uniformly on $A^\theta_{k,j}$ and $j$. It follows that
\begin{equation}\label{eq:P_{j,k}_typical_set}
    \lim_{k\to\infty}\E_k \bigl[\1_{A_{k,j}^\theta}|P_{j,k}-1|\bigr]=0.
\end{equation}
In particular, this implies
\begin{equation*}
    \lim_{k\to\infty}\E_k \bigl[\1_{A_{k,j}^\theta}P_{j,k}\bigr]=1.
\end{equation*}
Since by~\eqref{eq:expectation_Pj}
$1=\E_k \bigl[\1_{A_{k,j}^\theta}P_{j,k}\bigr]+
    \E_k \bigl[\1_{\Omega^k\setminus A_{k,j}^\theta}P_{j,k}\bigr],$
it follows that
\begin{equation}\label{eq:P_{j,k}_atypical_set}
    \lim_{k\to\infty}\E_k \bigl[\1_{\Omega^k\setminus A_{k,j}^\theta}P_{j,k}\bigr]=0.
\end{equation}

Under the measure $\mu^k$, the random variables $(\omega_i\gamma_{i+j-1})_{1\le i\le k}$ are independent with mean zero and variance $\gamma_{i+j-1}^2$. 
Hence, by Chebyshev's inequality and~\eqref{eq:decay}, we get
\[\SMALL
\mu^k\bigl(\Omega^k\setminus A_{k,j}^\theta\bigr)\le \bigl(\sum_{i=1}^k\gamma_{i+j-1}^2\bigr)^{2\theta}=
 O \bigl(k^{-4\delta\theta})\bigr)\longrightarrow 0,
\]
uniformly in $j\ge 2^{\epsilon k}$, as $k\to\infty$.
Applying equation~\eqref{eq:P_{j,k}_atypical_set},
\begin{align*}
    \E_k \bigl[\1_{\Omega^k\setminus A_{k,j}^\theta}|P_{j,k}-1|\bigr]&\le
    \E_k \bigl[\1_{\Omega^k\setminus A_{k,j}^\theta}P_{j,k}\bigr]+\mu^k(\Omega^k\setminus A_{k,j}^\theta)\longrightarrow 0.
\end{align*}
Combining the latter with~\eqref{eq:P_{j,k}_typical_set},
\[
\E_k [|P_{j,k}-1|]=
\E_k \bigl[\1_{A_{k,j}^\theta}|P_{j,k}-1|\bigr]+
\E_k \bigl[\1_{\Omega^k\setminus A_{k,j}^\theta}|P_{j,k}-1|\bigr]\longrightarrow 0,
\]
which finishes the proof.
\end{proof}
\end{prop}

\begin{remark}
    The exponent $1/2+\delta$ in the decay of $(\gamma_n)$ is  heuristically explained by applying of the central limit theorem to equation~\eqref{eq:P_{j,k}}. The sum of independent random variables $\sum_{i=1}^k\omega_i\gamma_{i+j-1}$  typically grows proportionally to $\bigl(\sum_{i=1}^k{\gamma_{i+j-1}^2}\bigr)^{1/2}$. Thus, the elements of $A^\theta_{k,j}$ characterize the asymptotics of $P_{j,k}$.
\end{remark}

We now can complete the proof that $C_{k}\rightarrow0$. 
For fixed
$k\ge1$ and $j\in[2^{k}]$, we let
\[
\eta_{j}=\E_k \bigl|\E_k [I_{j}|\sigma(I_{i}:i\in\Gamma_{j}^{w})]-2^{-k}\bigr|,
\]
where $\Gamma_j^w\subset[2^k]$ is from~\eqref{eq:strong_weak_dependence}. 
Consider now the  random variable $W(x,\omega)=\omega$ and let $\xi_{j}=\E_k [I_{j}-2^{k}|\mathcal{F}_{j}]$, where $\mathcal{F}_{j}=\sigma\bigl(\{I_{i}:i\in\Gamma_{j}^{w}\},W\bigr)$. Applying
 the tower property twice,
\[
    \eta_{j}=\E_k \Bigl|\E_k \bigl[\xi_{j}\bigl|\sigma(I_{i}:i\in\Gamma_{j}^{w})\bigr]\Bigr|\le\E_k \Bigl[\E_k \bigl[|\xi_{j}|\bigl|\sigma(I_{i}:i\in\Gamma_{j}^{w})\bigr]\Bigr]=\E_k |\xi_{j}|.
\]
Since $|j-i|\ge k$, the variable $I_j$ is independent of $(I_{i}:i\in\Gamma_{j}^{w})$ conditionally to $\{W=\omega\}$. Hence, by equation~\eqref{eq:P(Ij=1)}, 
\[
    \E_k [I_{j}|\mathcal{F}_{j}](x,\omega)=\P_k (I_{j}=1|W=\omega)
    =2^{-k}P_{j,k}(\omega).
\]
Therefore, $\xi_j=2^{-k}(P_{j,k}-1)$ and
\[\SMALL
C_{k} \le \sum_{j\in[2^{k}]}\E_k |\xi_{j}| = 2^{-k}\sum_{j\in[2^{k}]}\E_k 
\bigl|P_{j,k}-1\bigr|.
\]
By~\eqref{eq:expectation_Pj} we know that $\E_k [P_{j,k}]=1$, so as $k\to\infty$
\[\SMALL
    2^{-k}\sum_{j\le 2^{\epsilon k}}
    \E_k \bigl|P_{j,k}-1\bigr|\le 
    2\cdot 2^{-k(1-\epsilon)}=o(1).
\]
Hence, Proposition~\ref{prop:Pj_uniform} yields that
\[\SMALL
   C_k\le o(1)+2^{-k}\sum_{2^{\epsilon k}\le j\le 2^k}\E_k 
    \bigl|P_{j,k}-1\bigr|\to 0.
\]
This concludes the estimate for $C_{k}$ and thus
our proof of Proposition~\ref{prop:annealed}.

\section{Non-convergence}\label{sec:non_conv}

Without loss of generality, we fix $\delta\in(0,1/2), n_0\ge1$, and assume that
\[
    \gamma_{n}=\log^{-(1/2-\delta)}n,
    \qquad \text{for all } n\ge n_0.
\]
We consider $M_k$ defined in~\eqref{eq:Mk} on the probability space  $(\Omega^\N\times\Omega^k,\P_k )$; we shall show that~$M_{k}$  does \textbf{not} converge in distribution to a Poisson random variable with mean one. 
In the current section we prove this result in the annealed setting, whereas the second part of Theorem~\ref{thm:main_thm} addresses the quenched result. But since quenched convergence implies annealed convergence, this is sufficient.

Before proving the annealed case, we need to establish a few preliminary results. 
Let~$k\in\mathbb{N}$ and 
let $D_{+},D_{-}\subseteq\{1,\dots,k\}$ be sets of equal
size. For $j\ge1$, write
\[
    \Xi_j=\Xi_j(D_{+},D_{-})=\prod_{i\in D_{+}}(1+\gamma_{i+j-1})\prod_{i\in D_{-}}(1-\gamma_{i+j-1}).
\]

\begin{prop}
\label{prop:balanced-products}For any $\epsilon\in(0,1)$ there is $k_0\ge1$ such that $\Xi_j\leq1$
uniformly in $k\ge k_0$, $2^{\epsilon k}\le j\le 2^k$, and  $D_{+}$, $D_{-}$.

\begin{proof}
Let $\ell=|D_{+}|=|D_{-}|\leq k$. Because $\gamma_{n}$ is decreasing,
the product defining $\Xi_j$ can only increase if we replace each $1+\gamma_{i+j-1}$
by $1+\gamma_{j}$, and each $1-\gamma_{i+j-1}$ by $1-\gamma_{j+k}$.
Thus,
\[
    \Xi_j\leq(1+\gamma_{j})^{\ell}(1-\gamma_{j+k})^{\ell}=(1+\gamma_{j}-\gamma_{j+k}-\gamma_{j}\gamma_{j+k})^{\ell}.
\]
Let $f(x)=\log^{-(1/2-\delta)}x$, $x>0$, so that $f(n)=\gamma_n$, $n\ge1$.
Since $f$ is deceasing and concave, for $x<y$ we have $|f(x)-f(y)|\leq|x-y||f'(x)|$.
Applying this with $x=j$ , $y=j+k$, and using $j\geq2^{\varepsilon k}$
, $f'(x)=-c(1/2-\delta)(x\log^{3/2-\delta}x)^{-1}$, we get
\begin{equation*}\label{eq:derivative}
    \gamma_{j}-\gamma_{j+k}\leq k\cdot|f'(j)|= O(2^{-\varepsilon k}\cdot k^{-(1/2-\delta)}).
\end{equation*}
On the other hand, using $j,j+k\leq2^{k}+k<2^{k+1}$ for all $k$ sufficiently large, we have
$\gamma_{j}\gamma_{j+k}\geq c^2(\log2/(k+1))^{1-2\delta}$. It follows that
$1+\gamma_{j}-\gamma_{j+k-1}-\gamma_{j}\gamma_{j+k}<1$, and the same holds after raising to the $\ell$-th power,
giving us $\Xi_j$$\leq1$. This proves the statement.
\end{proof}

\end{prop}

For $\eta>0$ and $k\ge1$, define
\begin{equation}\label{eq:atypical_omegas}\SMALL
\Omega_k^\eta  =\{\omega\in\Omega^k:\sum_{i=1}^{k}\omega_{i}<-\eta\sqrt{k}\}.
\end{equation}
When convenient, we identify $\Omega_k^\eta\subseteq\Omega^k$
with its lift $\{(x,\omega)\mid\omega\in\Omega_k^\eta\}$ to $\Omega^{\N}\times\Omega^k$. 
\begin{lem}\label{lem:bound2} 
$\P_k (\Omega_k^\eta\cap\{M_{k}\geq1\})\rightarrow0$ as $k\to\infty$.

\begin{proof}
By Fubini, it suffices to bound $\nu(x:M_{k}(x,\omega)\ge1)=\P_k (M_{k}\geq1|\{\omega\})$
uniformly in $\omega\in\Omega_k^\eta$. Since $M_{k}=\sum_{j\in[2^{k}]}I_{j}$, we get by\eqref{eq:P(Ij=1)} that for all $\omega\in\Omega^k$
\begin{equation*}\SMALL
    \nu(x:M_{k}(x,\omega)\ge1)   \leq \sum_{j\in[2^{k}]}\nu(x:I_{j}(x,\omega)=1) = 2^{-k}\sum_{j\in[2^{k}]}P_{j,k}(\omega)\label{eq:bound2-sum}
\end{equation*}

Let $\epsilon\in(0,1)$. We first claim that the sum on the right changes by $o(1)$ if we
sum over $2^{\epsilon k}\leq j\leq2^{k}$ instead of $1\leq j\leq2^{k}$.
Indeed, using $\gamma_{n}\rightarrow0$, there is $j_{0}$ such that
$1+2\gamma_{j}<2^{(1-\varepsilon)/2}$ for any $j\ge j_0$. By the fact that $\gamma_{n}\rightarrow0$,
for every fixed $j\in\mathbb{N}$ we have $\sup_{\omega\in\Omega^k}2^{-k}P_{j,k}(\omega)=o(1)$
as $k\rightarrow\infty$, so 
\[\SMALL
    \sum_{1\leq j\leq j_{0}}2^{-k}P_{j,k}(\omega)=j_{0}\cdot o(1)=o(1).
\]
Also, for all $j\ge j_0$,
\begin{align*}\SMALL
    P_{j,k}(\omega)   = \prod_{i=1}^{k}(1+2\omega_{i}\gamma_{ i+j-1}) \leq 
   2^{(1-\varepsilon)k/2},
\end{align*}
so
\[\SMALL
    2^{-k}\sum_{j_{0}\leq j\leq2^{\varepsilon k}}P_{j,k}(\omega)<2^{-k}\cdot2^{\varepsilon k}\cdot2^{(1-\varepsilon)k/2}=o(1),
\]
uniformly in $\omega\in\Omega^k$. 
Thus, we have shown that
\[\SMALL
    \nu(x:M_{k}(x,\omega)\ge1)=o(1)+2^{-k}\sum_{2^{\epsilon k}\le j\le 2^k}P_{j,k}(\omega).
\]

Let $  N_+(\omega)  =\#\{1\leq i\leq k\mid\omega_{i}=1\}$,
and let $D_{+},D_{-}\subseteq[k]$ denote the sets of positions of
the first $N_+(\omega)$ occurrences of $+1,-1$ in $\omega$, respectively.
Since $\sum_{i\in D_+\cup D_-}\omega_i=0$,  the set 
$E(\omega)=[k]\setminus(D_{+}\cup D_{-})$ 
has cardinality $|E|=\bigr|\sum_{i=1}^k\omega_i\bigr|$. Let now  $\omega\in\Omega_k^\eta$. It follows that
$\omega_{i}=-1$ for $i\in E$ and $|E|>\eta\sqrt{k}$. Since $(\gamma_n)$ is decreasing, by Proposition~\ref{prop:balanced-products},
\begin{align*}\SMALL
    P_{j,k}(\omega)=\Xi_{j}(D_{+},D_{-})\cdot\prod_{i\in E(\omega)}(1-2\gamma_{ i+j-1})\le 
    (1-2\gamma_{k+2^{k}})^{|E|},
\end{align*}
for all $k\ge1$ sufficiently large, uniformly in 
$2^{\epsilon k}\leq j\leq2^{k}$ and  $\omega\in\Omega_k^\eta$.
By $|E|>\eta\sqrt{k}$,
\begin{align*}
    2^{-k}\sum_{2^{\varepsilon k}\le j\le 2^k}P_{j,k}(\omega) & <(1-2\gamma_{k+2^{k}})^{\eta k^{1/2}}\leq\biggl(1-\frac{c'}{k^{1/2-\delta}}\biggr)^{\eta k^{1/2}},
\end{align*}
for some $c'>0$. 
Since the exponent tends to infinity faster than the denominator, the last expression tends to zero as
$k\rightarrow\infty$, as desired.
\end{proof}
\end{lem}

If $Y$ is Poisson with parameter $1$ then $\P_k (Y=0)=1/e$.
Thus, the next proposition shows that $M_{k}$ does not converge in
distribution to $\Po(1)$.
\begin{prop}
$\limsup_{k\rightarrow\infty}\P_k (M_{k}=0)>1/e$.

\begin{proof}
Since $M_{k}\geq0$ is integer-valued, the complement of the event
$\{M_{k}=0\}$ is $\{M_{k}\geq1\}$; we shall bound the probability of the
latter event from above. For a parameter $\eta>0$ that we shall choose
later, let $\Omega_k^\eta$ be as in~\eqref{eq:atypical_omegas} and let $\mathcal{N}$ be
a standard Gaussian. Since on the space $(\Omega^k,\mu^{k})$ the
random variables $\{\omega_i\}_{1\le i\le k}$ are i.i.d. with unitary second moment,
as $k\to\infty$ the Central limit theorem yields that
\[
\P_k ((\Omega_k^\eta)^{c})=\P_k (\mathcal{N}\ge-\eta)+o(1).
\]
Therefore, by Lemma \ref{lem:bound2},
\[
\begin{split}
    \P_k (M_{k}\ge1) &  \leq \P_k (\Omega_k^\eta\cap\{M_{k}\geq1\})+\P_k ((\Omega_k^\eta)^{c}) = o(1)+\P_k (\mathcal{N}\ge-\eta)\end{split}
\]
Since $\lim_{\eta\rightarrow0}\P_k (\mathcal{N}\geq-\eta)=\P_k (\mathcal{N}\geq0)=1/2$,
by choosing $\eta$ small enough we can ensure that $\P_k (\mathcal{N}\geq-\eta)<1-1/e$.
It then follows that
\[
    \limsup_{k\rightarrow\infty}\P_k (M_{k}\geq1)<1-\frac{1}{e},
\]
as desired.
\end{proof}
\end{prop}

\subsection*{Acknowledgments}
The authors would like to thank Zemer Kosloff for insightful discussions throughout the development of this work, in particular suggesting a simplification to the proof of Proposition~\ref{prop:Pj_uniform}. This research was supported by the Israel Science Foundation Grant 3056/21.

\bibliographystyle{apalike.bst}
\bibliography{References}

\end{document}